\documentclass{article}   	

\usepackage{microtype}
\usepackage{graphicx}
\usepackage{subfigure}
\usepackage{booktabs} 

\usepackage[accepted]{icml2019}

\usepackage{amssymb,amsmath,amsthm}
\usepackage{color}

\newtheorem{Ass}{Assumption}

\newtheorem{Thm}{Theorem}
\newtheorem{Lem}{Lemma}
\newtheorem{Cor}{Corollary}

\newtheorem{Fact}{Fact}
\newtheorem{Rem}{Remark}

\newcommand{\argmin}[1]{\underset{#1}{\mbox{argmin }}}

\newcommand{\norm}[1]{\Vert{#1}\Vert}

\newcommand{\inprod}[1]{\langle #1\rangle}

\newcommand{\mb}[1]{{\mathbf{#1} }}
\newcommand{\mbb}[1]{{\mathbb{#1} }}
\newcommand{\mc}[1]{{\mathcal{#1} }}

\newcommand*{\tran}{^{\mkern-1.5mu\mathsf{T}}}

\newcommand*{\defeq}{\overset{\Delta}{=}}

\icmltitlerunning{Parallel SGD with Dynamic Batch Sizes for Stochastic Non-Convex Optimization}

\begin{document}

\twocolumn[
\icmltitle{On the Computation and Communication Complexity of Parallel SGD with Dynamic Batch Sizes for Stochastic Non-Convex Optimization}
\icmlsetsymbol{equal}{*}           
\begin{icmlauthorlist}
\icmlauthor{Hao~Yu}{alibaba}
\icmlauthor{Rong~Jin}{alibaba}
\end{icmlauthorlist}

\icmlaffiliation{alibaba}{Machine Intelligence Technology Lab, Alibaba Group (U.S.) Inc., Bellevue, WA}

\icmlcorrespondingauthor{Hao Yu}{eeyuhao@gmail.com}
\icmlkeywords{Stochastic Optimization}

\vskip 0.3in
]

\printAffiliationsAndNotice{}

\begin{abstract}
For SGD based distributed stochastic optimization, computation complexity, measured by the convergence rate in terms of the number of stochastic gradient calls, and communication complexity, measured by the number of inter-node communication rounds, are two most important performance metrics.  The classical data-parallel implementation of SGD over $N$ workers can achieve linear speedup of its convergence rate but incurs an inter-node communication round at each batch. We study the benefit of using dynamically increasing batch sizes in parallel SGD for stochastic non-convex optimization by charactering the attained convergence rate and the required number of communication rounds.  We show that for stochastic non-convex optimization under the P-L condition, the classical data-parallel SGD with exponentially increasing batch sizes can achieve the fastest known $O(1/(NT))$ convergence with linear speedup using only $\log(T)$ communication rounds. For general stochastic non-convex optimization, we propose a Catalyst-like algorithm to achieve the fastest known $O(1/\sqrt{NT})$ convergence with only $O(\sqrt{NT}\log(\frac{T}{N}))$ communication rounds.  
\end{abstract}

\section{Introduction}

Consider solving the following stochastic optimization
\begin{align}
\min_{ \mb{x}\in \mbb{R}^{m}}  \quad f(\mb{x})  \defeq \mbb{E}_{\zeta\sim \mc{D}}[F(\mb{x}; \zeta)] \label{eq:sto-opt}
\end{align}
with a fixed yet unknown distribution $\mc{D}$ only by accessing i.i.d. stochastic gradients $\nabla F(\cdot; \zeta)$. Most machine learning applications can be cast into the above stochastic optimization where $\mb{x}$ refers to the machine learning model, random variables $\zeta \sim \mc{D}$ refer to instance-label pairs and $F(\mb{x}; \zeta)$ refers to the corresponding loss function. For example, consider a simple least squares linear regression problem: let $\zeta_i = (\mb{a}_i, b_i)\in \mc{D}$ be training data collected offline or online\footnote{Note that if the training data is from a finite set collected offline, the stochastic optimization can also be written as a finite sum minimization, which is a special case of the stochastic optimization with known uniform distribution $\mc{D}$. However, for online training, since $(\mb{a}_i, b_i)$ is generated gradually and disclosed to us one by one, we need to solve the more challenging stochastic optimization with unknown distribution $\mc{D}$. The algorithms developed in this paper does not requires any knowledge of distribution $\mc{D}$.}, where each $\mb{a}_i$ is a feature vector and $b_i$ is its label, then $F(\mb{x}; \zeta_i) = \frac{1}{2} (\mb{a}_i\tran \mb{x} - b_i)^2$.  Throughout this paper, we have the following assumption: 
\begin{Ass}\label{ass:basic}~
\vspace {-7pt}
\begin{enumerate}
\item {\bf Smoothness:} The objective function $f(\mb{x})$ in problem \eqref{eq:sto-opt} is smooth with modulus $L$.
\item {\bf Unbiased gradients with bounded variances:} Assume there exits a stochastic first-order oracle (SFO) to provide independent unbiased stochastic gradients $\nabla F(\mb{x}; \zeta)$ satisfying $$\mbb{E}_{\zeta\sim \mc{D}}[\nabla F(\mb{x}; \zeta)] = \nabla f(\mb{x}),  \forall \mb{x}.$$ The  unbiased stochastic gradients have a bounded variance, i.e., there exits a constant $\sigma > 0$ such that 
\begin{align}
\mbb{E}_{\zeta\sim \mc{D}} \norm{\nabla F(\mb{x}; \zeta) - \nabla f(\mathbf{x})}^{2} \leq \sigma^{2}
\end{align}
\end{enumerate}
\end{Ass}
When solving  stochastic optimization \eqref{eq:sto-opt} only with sampled stochastic gradients, the computation complexity, which is also known as the convergence rate, is measured by the decay law of the solution error with respect to the number of access of the {\bf stochastic first-order oracle (SFO)} that provides sampled stochastic gradients  \cite{book_Nemirovsky83, Ghadimi16MP}. For strongly convex stochastic minimization, SGD type algorithms \cite{Nemirovski09SIOPT, Hazan14JMLR, Rakhlin12ICML} can achieve the optimal $O(1/T)$ convergence rate. That is, the error is ensured to be at most $O(1/T)$ after $T$ access of stochastic gradients. For non-convex stochastic minimization, which is  the case of training deep neural networks, SGD type algorithms can achieve an $O(1/\sqrt{T})$ convergence rate\footnote{For general non-convex functions, the convergence rate is usually measured in terms of $\norm{\nabla f(\mb{x})}^2$ which in some sense can be considered as  the counterpart of $f(\mb{x}) - f(\mb{x}^\ast)$ in convex case \cite{book_ConvexOpt_Nesterov,GhadimiLan13SIOPT}.}. Classical SGD type algorithms can be accelerated by utilizing multiple workers/nodes to follow a {\bf parallel SGD (PSGD)} procedure where each worker computes local stochastic gradients in parallel, aggregates all local gradients, and updates its own local solution using the average of all gradients. Such a data-parallel training strategy with $N$ workers has $O(1/(NT))$ convergence for strongly convex minimization and  $O(1/\sqrt{NT})$ convergence for smooth non-convex stochastic minimization, both of which is $N$ times faster than SGD with a single worker \cite{Dekel12JMLR, GhadimiLan13SIOPT, Lian15NIPS}.  This is known as the linear speedup\footnote{The linear speedup property is desirable for parallel computating algorithms since it means the algorithm's computation capability can be expanded with perfect horizontal scalability. } (with respect to the number of nodes) property of PSGD.  

However, such linear speedup is often not attainable in practice because PSGD involves additional coordination and communication cost as most other distributed/parallel algorithms do. In particular, PSGD requires aggregating local batch gradients among all workers after evaluations of local batch SGD. The corresponding communication cost for gradient aggregations is quite heavy and often becomes the performance bottleneck.  

Since the number of inter-node communication rounds in PSGD over multiple nodes is equal to the number of batches, it is desirable to use larger batch sizes to avoid communication overhead as long as the large batch size does not damage the overall computation complexity (in terms of number of access of SFO).  For training deep neural networks, practitioners have observed that SGD using dynamically increasing batch sizes can converges to similar test accuracy with the same number of epochs but significantly fewer number of batches when compared with SGD with small batch sizes \cite{Devarakonda17ArXiv,Smith18ICLR}. The idea of using large or increasing batch sizes can be partially backed by some recent theoretical works \cite{Bottou18SIAM, DE17AISTATS}. It is shown in \cite{DE17AISTATS} that if the batch size is sufficiently large such that the randomness, i.e., variances, is dominated by gradient magnitude, then SGD essentially degrades to deterministic gradient descent. However, in the worst case, e.g., stochastic optimization \eqref{eq:sto-opt} or large-scale optimization with limited budgets of SFO access,  SGD with large batch sizes considered in \cite{DE17AISTATS} can have worse convergence performance than SGD with fixed small batch sizes \cite{Bottou08NIPS, Bottou18SIAM}. For {\bf strongly convex} stochastic minimization, it is proven in \cite{Friedlander12SIOPT, Bottou18SIAM} that SGD with exponentially increasing batch sizes can achieve the same $O(1/T)$ convergence as SGD with fixed small batch sizes, where $T$ is the number of access of SFO.  The results in \cite{Friedlander12SIOPT, Bottou18SIAM} are encouraging since it means using exponentially increasing batch sizes can preserve the low $O(1/T)$ computation complexity with $\log(T)$ communication complexity that is significantly lower than $O(T)$ required by SGD with fixed batch sizes for distributed strongly convex stochastic minimization. However, the computation and communication complexity remains under-explored for distributed stochastic non-convex optimization, which is the case of training deep neural networks. While work \cite{SmithLe18ICLR, Smith18ICLR} justify SGD with increasing batch sizes by relating it with the integration of a stochastic differential equation for which decreasing learning rates can roughly compensate the effect of increasing batch sizes, rigorous theoretical characterization on its computation and communication complexity (as in \cite{Nemirovski09SIOPT,Bottou18SIAM}) is missing for stochastic non-convex optimization. In general, it remains unclear {\it ``If using dynamic batch sizes in parallel SGD can yield the same fast $O(1/\sqrt{NT})$ convergence rate (with linear speedup with respect to the number of nodes) as the classical PSGD for non-convex optimization?"} and {\it ``What is the corresponding communication complexity of using dynamic batch sizes to solve distributed non-convex optimization?"}

{\bf Our Contributions:} This paper aims to characterize both computation and communication complexity when using the idea of dynamically increasing batch sizes in SGD to solve stochastic non-convex optimization with $N$ parallel workers.  We first consider non-convex optimization satisfying the Polyak-Lojasiewicz (P-L) condition, which can be viewed as a generalization of strong convexity for non-convex optimization.  We show that by simply exponentially increasing the batch sizes at each worker (formally described in Algorithm \ref{alg:new}) in the classical data-parallel SGD, we can solve non-convex optimization with the fast $O(1/(NT))$ convergence using only $O(\log(T))$ communication rounds.  For general stochastic non-convex optimization (without P-L condition), we propose a Catalyst-like \cite{Lin15NIPS,Paquette18AISTATS} approach (formally described in Algorithm \ref{alg:new-catalyst}) that wraps Algorithm \ref{alg:new} with an outer loop that iteratively introduces auxiliary problems. We show that Algorithm \ref{alg:new-catalyst} can solves general stochastic non-convex optimization with $O(1/\sqrt{NT})$ computation complexity and $O(\sqrt{TN}\log(\frac{T}{N}))$ communication complexity.  In both cases, using dynamic batch sizes can achieve the linear speedup of convergence with communication complexity less than that of existing communication efficient parallel SGD methods with fixed batch sizes \cite{Stich18ArXiv, Yu18ArXivAAAI}.

\section{Non-Convex Minimization Under the P-L Condition} \label{sec:PL-condition}
This section considers problem \eqref{eq:sto-opt} satisfying the Polyak-Lojasiewicz (P-L) condition defined in Assumption \ref{ass:P-L}.
\begin{Ass}\label{ass:P-L}
The objective function $f(\mb{x})$ in problem \eqref{eq:sto-opt} satisfies the Polyak-Lojasiewicz (P-L) condition with modulus $\mu > 0$. That is, 
\begin{align}
\frac{1}{2}\norm{\nabla f(\mb{x})}^2 \geq \mu (f(\mb{x}) -  f^\ast), \forall \mb{x}
\end{align}
where $f^\ast$ is the global minimum in problem \eqref{eq:sto-opt}.
\end{Ass} 

The P-L condition is originally introduced by Polyak in \cite{Polyak63} and holds for many machine learning models. Neither the convexity of $f(\mb{x})$ nor the uniqueness of its global minimizer is required in the P-L condition. In particular, the P-L condition is weaker than many other popular conditions, e.g., strong convexity and the error bound condition, used in optimization literature \cite{Karimi16}.  See e.g. Fact \ref{fact:sc-imply-pl}.

\begin{Fact}[Appendix A in \cite{Karimi16}] \label{fact:sc-imply-pl}
	If smooth function $\phi: \mbb{R}^m\mapsto \mbb{R}$ is strongly convex with modulus $\mu>0$, then it satisifes the P-L condition with the same modulus $\mu$.
\end{Fact}

One important example is: $ f(\mb{x}) = g(\mb{Ax})$ with strongly convex $g(\cdot)$ and possibly rank deficient matrix $\mb{A}$, e.g. $f(\mb{x}) = \norm{\mb{Ax} -\mb{b}}^2$ used in least squares regressions. While $f(\mb{x}) =g(\mb{Ax})$ is not strongly convex when $\mb{A}$ is rank deficient, it turns out that such $f(\mb{x})$ always satisfies the P-L condition \cite{Karimi16}.

Consider the {\bf C}ommunication {\bf R}educed {\bf P}arallel {\bf S}tochastic {\bf G}radient {\bf D}escent ({\bf CR-PSGD}) algorithm described in Algorithm \ref{alg:new}. The inputs of CR-PSGD are:  (1) $N$,  the number of parallel workers; (2) $T$,  the total number of gradient evaluations at each worker; (3) $\mb{x}_1$, the common initial point at each worker;  (3) $\gamma > 0$, the learning rate; (4) $B_1$, the initial SGD batch size at each worker; (5) $\rho > 1$, the batch size scaling factor. Compared with the classical PSGD, our CR-PSGD has the minor change that each worker exponentially increases its own SGD batch size with a factor $\rho$.  Since $B_t$ increasingly exponentially, it is easy to see that the ``while" loop in Algorithm \ref{alg:new} terminates after at most $O(\log T)$ steps. Meanwhile, we note that inter-worker communication is used only to aggregate individual batch SGD averages and happens only once in each ``while" loop iteration. As a consequence, CR-PSGD only involves $O(\log T)$ rounds of communication. The remaining part of this section further proves that CR-PSGD has $O(1/(NT))$ convergence. 

Similar ideas of exponentially increasing batch size appear in other works, e.g., \cite{Hazan14JMLR, Zhang13ICML}, for different purposes and with different algorithm dynamics. In this paper, we explore this idea in the context of parallel stochastic optimization. It is impressive that such a simple idea enables us to obtain a parallel algorithm to achieve the fast $O(1/(NT))$ convergence with only $O(\log T)$ rounds of communication  for stochastic optimization under the P-L condition. When considering stochastic strongly convex minimization that is a subclass of stochastic optimization under the P-L condition,  the $O(\log T)$ communication complexity attained by our CR-PSGD is significantly less than the $O(\sqrt{NT})$ communication complexity attained by the local SGD method in \cite{Stich18ArXiv}.

\begin{algorithm}[tb]
	\caption{ $\text{CR-PSGD}(f, N, T, \mb{x}_1, B_1, \rho, \gamma)$}\label{alg:new}
	\begin{algorithmic}[1]
		\STATE {\bf Input:} $N$, $T$, $\mb{x}_1 \in \mathbb{R}^m$, $\gamma $ , $B_1$ and $\rho > 1$.
		\STATE Initialize $t=1$
		\WHILE{$\sum_{\tau=1}^t B_\tau \leq T$}
		\STATE Each worker $i$ observes $B_t$ unbiased i.i.d. stochastic gradients at point $\mb{x}_{t}$ given by $\mb{g}_{i,j} \defeq \nabla F(\mb{x}_{t}; \zeta_{i,j}), j\in\{1,\ldots,B_t\}, \zeta_{i,j}\sim \mc{D}$ and calculates its batch SGD average $\bar{\mb{g}}_{t,i} = \frac{1}{B_t}\sum_{j=1}^{B_t} \mb{g}_{i,j}$.
		\STATE Aggregate all $\bar{\mb{g}}_{t,i}$ from $N$ workers and compute their average $\bar{\mb{g}}_{t} = \frac{1}{N} \sum_{i=1}^{N} \bar{\mb{g}}_{t,i}$.
		\STATE Update $\mb{x}_{t+1}$ over all $N$ workers in parallel via: $\mb{x}_{t+1} = \mb{x}_{t} - \gamma \bar{\mb{g}}_{t}$.
		\STATE Set $B_{t+1} = \lfloor \rho^t B_{1} \rfloor$ where $\lfloor z\rfloor$ represents the largest integer no less than $z$.
		\STATE Update $t \leftarrow t+1$.
		\ENDWHILE
		\STATE {\bf Return:} $\mathbf{x}_t$
	\end{algorithmic}
\end{algorithm}

The next simple lemma relates per-iteration error with the batch sizes and is a key property to establish the convergence rate of Algorithm \ref{alg:new}.
\begin{Lem}\label{lm:PL-one-step}
Consider problem \eqref{eq:sto-opt} under Assumptions \ref{ass:basic}-\ref{ass:P-L}. If we choose $\gamma < \frac{1}{L}$ in Algorithm \ref{alg:new}, then for all $t\in\{1,2,\ldots,\}$, we have
{\small 
\begin{align}
    \mbb{E}[f(\mb{x}_{t+1}) - f^\ast)] \leq (1-\nu) \mbb{E}[f(\mb{x}_{t}) - f^\ast] + \frac{\gamma (2-L\gamma)}{2NB_t} \sigma^2 \label{eq:lm-PL-one-step}
\end{align}
}%
where $f^\ast$ is the global minimum in problem \eqref{eq:sto-opt} and $\nu \defeq  \frac{1}{2}\gamma \mu (1- L\gamma)$ satisfies $0<\nu <1$. 

\end{Lem}
\begin{proof}
Fix $t\geq 1$. By the smoothness of $f(\mb{x})$ in Assumption \ref{ass:basic}, we have
{\footnotesize
\begin{align}
   &f(\mb{x}_{t+1}) \nonumber\\
    \leq & f(\mb{x}_{t}) + \inprod{\nabla f(\mb{x}_{t}), \mb{x}_{t+1} - \mb{x}_{t}} + \frac{L}{2} \norm{\mb{x}_{t+1} - \mb{x}_{t}}^2 \nonumber\\
    \overset{(a)}{=} & f(\mb{x}_{t}) - \gamma \inprod{\nabla f(\mb{x}_{t}), \bar{\mb{g}}_{t}} + \frac{L}{2}\gamma^2 \norm{\bar{\mb{g}}_{t}}^2 \nonumber\\
    = & f(\mb{x}_{t}) + \gamma \inprod{\bar{\mb{g}}_{t} - \nabla f(\mb{x}_{t}), \bar{\mb{g}}_{t}} - \gamma \norm{\bar{\mb{g}}_{t}}^2+ \frac{L}{2}\gamma^2  \norm{\bar{\mb{g}}_{t}}^2 \nonumber
        \end{align}
   \begin{align}
    \overset{(b)}{\leq}& f(\mb{x}_{t}) + \frac{\gamma}{2}\norm{\bar{\mb{g}}_{t} - \nabla f(\mb{x}_{t})}^2 + \frac{\gamma}{2}( L\gamma -1) \norm{\bar{\mb{g}}_{t}}^2 \nonumber\\
    \overset{(c)}{\leq}& f(\mb{x}_{t})  + \frac{\gamma}{4}( L\gamma -1)\norm{\nabla f(\mb{x}_{t})}^2 + \frac{\gamma}{2}(2-L\gamma) \norm{\bar{\mb{g}}_{t} - \nabla f(\mb{x}_{t})}^2 \nonumber\\
    \overset{(d)}{\leq}& f(\mb{x}_{t})  + \frac{1}{2}\gamma \mu( L\gamma -1)(f(\mb{x}_{t}) -f^\ast)  \nonumber \\ &\quad + \frac{\gamma}{2}(2-L\gamma) \norm{\bar{\mb{g}}_{t} - \nabla f(\mb{x}_{t})}^2 \label{eq:pf-lm-eq1}
\end{align}
}%
where (a) follows by substituting $\mb{x}_{t+1} = \mb{x}_{t} - \gamma \bar{\mb{g}}_{t}$; (b) follows by applying elementary inequality $\inprod{\mb{u}, \mb{v}} \leq \frac{1}{2}\norm{\mb{u}}^2 + \frac{1}{2}\norm{\mb{v}}^2$ with $\mb{u} =\bar{\mb{g}}_{t} - \nabla f(\mb{x}_{t})$ and $\mb{v} =\bar{\mb{g}}_{t}$; (c) follows by noting that $L\gamma -1 <0$ under our selection of $\gamma$ and applying elementary inequality $\norm{\mb{u} + \mb{v}}^2 \geq \frac{1}{2}\norm{\mb{u}}^2 -\norm{\mb{v}}^2$ with $\mb{u} = \nabla f(\mb{x}_{t})$ and $\mb{v} = \bar{\mb{g}}_{t} - \nabla f(\mb{x}_{t})$; and (d) follows by noting that $\gamma(L\gamma -1) <0$ under our selection of $\gamma$ and $\norm{\nabla f(\mb{x}_{t})}^2 \geq 2\mu(f(\mb{x}_{t}) -f^\ast)$ by Assumption \ref{ass:P-L}. 

Defining $\nu \overset{\Delta}{=} \frac{1}{2}\gamma \mu (1- L\gamma) $, subtracting $f^\ast$ from both sides of \eqref{eq:pf-lm-eq1}, and rearranging terms yields
\begin{align}
    &f(\mb{x}_{t+1}) - f^\ast \nonumber\\
     \leq& (1-\nu) (f(\mb{x}_{t}) -f^\ast) + \frac{\gamma}{2}(2-L\gamma) \norm{\bar{\mb{g}}_{t} - \nabla f(\mb{x}_{t})}^2
\end{align}
Taking expectations on both sides and noting that $\mbb{E}[\norm{\bar{\mb{g}}_t - \nabla f(\mb{x}_{t})}^2] \leq \frac{1}{NB_t}\sigma^2$, which further follows from Assumption \ref{ass:basic}  and the fact that each $\bar{\mb{g}}_t$ is the average of $NB_t$ i.i.d. stochastic gradients evaluated at the same point, yields
\begin{align*}
    \mbb{E}[f(\mb{x}_{t+1}) - f^\ast] \leq (1-\nu) \mbb{E}[f(\mb{x}_{t}) - f^\ast] + \frac{\gamma (2-L\gamma)}{2NB_t} \sigma^2
\end{align*}

It remains to verify why $0< \nu <1$.  Since $\gamma < \frac{1}{L}$, it is easy to see $\nu > 0$. Next, we show  $\frac{1}{2}\gamma\mu(1-L\gamma) <1$. By the smoothness of $f(\mb{x})$ (and Fact \ref{fact:smooth-grad-norm-ub} in Supplement \ref{app:basic}), we have 
\begin{align}
\frac{1}{2}\norm{\nabla f(\mb{x})}^2 \leq L(f(\mb{x}) - f^\ast), \forall \mb{x} \label{eq:pf-lm-eq2}
\end{align}
By Assumption \ref{ass:P-L}, we have 
\begin{align}
\frac{1}{2}\norm{\nabla f(\mb{x})}^2 \geq \mu (f(\mb{x}) -  f^\ast), \forall \mb{x} \label{eq:pf-lm-eq3}
\end{align}
Inequalities \eqref{eq:pf-lm-eq2} and \eqref{eq:pf-lm-eq3} together imply that $\mu \leq L$, which further implies that $\frac{1}{2} \gamma\mu(1-L\gamma) \leq \frac{1}{2} \gamma L(1-L\gamma) <1$.
\end{proof}

\begin{Rem}
	Note that by adapting steps (b) and (c) of \eqref{eq:pf-lm-eq1} in the proof of Lemma \ref{lm:PL-one-step}, i.e., using inequalities with slightly different coefficients for the squared norm terms, we can obtain \eqref{eq:lm-PL-one-step} with different $\nu$ values. Larger $\nu$ variants (with possibly more stringent conditions on the selection rule of $\gamma$) may lead to faster convergence (but with the same order) of Algorithm \ref{alg:new}. This paper does not explore further in this direction since the current simple analysis is already sufficient to provide the desired order of convergence/communication. The suggested finer development on $\nu$ can improve the constant factor in the rates but does not improve their order. Nevertheless, it is worthwhile to point out that the finer development on $\nu$ can be helpful to guide practitioners to tune Algorithm \ref{alg:new} according to their specific minimization problems.
\end{Rem} 

The $O(\frac{1}{NT})$ convergence with $O(\log T)$ communication rounds is summarized in Theorem \ref{thm:PL-rate}.

\begin{Thm} \label{thm:PL-rate}
Consider problem \eqref{eq:sto-opt} under Assumptions \ref{ass:basic}-\ref{ass:P-L}. Let $T>0$ be a given constant. If we choose $B_1 \geq 2$, $\gamma < \frac{1}{L}$ and $1 < \rho < \frac{1}{1-\nu}$, where\footnote{It is shown at the bottom of the proof for Lemma \ref{lm:PL-one-step} that $\nu$ is ensured to satisfy $0 < \nu < 1$ under the selection $\gamma < \frac{1}{L}$.} $\nu \defeq \frac{1}{2}\gamma \mu (1- L\gamma)$,  in Algorithm \ref{alg:new}, then the final output $\mb{x}_t$ returned by Algorithm \ref{alg:new} satisfies 
\begin{align}
\mbb{E}[f(\mb{x}_t) - f^\ast] \leq&  \frac{c_1 (f(\mb{x}_{1}) - f^\ast)}{T^{1+\delta}} + \frac{c_2}{NT} \nonumber\\
=&  O(\frac{1}{T^{1+\delta}}) + O(\frac{1}{NT})
\end{align}
where  $\delta \defeq \log_{\rho}(\frac{1}{1-\nu}) -1 >0$, $c_1 \defeq  \frac{1}{1-\nu} \big(\frac{B_1}{\rho -1}\big)^{1+\delta}$, $c_2 \defeq \frac{\rho^2\gamma(2-L\gamma))\sigma^2}{(1 - (1-\nu)\rho)(\rho - 1)}$, and $f^\ast$ is the minimum value of problem \eqref{eq:sto-opt}.
\end{Thm}
\begin{proof}
See Supplement \ref{sec:pf-PL-rate}.
\end{proof}

\begin{Rem}
Since $\delta > 0$, $O(\frac{1}{T^{1+\delta}})$ decays faster than $O(\frac{1}{NT})$ when $T$ is sufficiently large.  In fact, we can even explicitly choose suitable $\rho$ to make $\delta$ sufficiently large, e.g., we can choose $1 < \rho < \sqrt{\frac{1}{1-\nu}}$ to ensure $\delta>1$ such that $O(\frac{1}{T^{1+\delta}}) <O(\frac{1}{T^2})$. In this case, as long as $T\geq N$, which is almost always true in practice, the error term on the right side of \eqref{eq:pf-thm-PL-eq2} has order $O(\frac{1}{NT})$.
\end{Rem}

Recall that if $f(\mb{x})$ is strongly convex with modulus $\mu$, then it satisfies Assumption \ref{ass:P-L} with the same $\mu$ by Fact \ref{fact:sc-imply-pl}. Furthermore, if $f(\mb{x})$ is strongly convex with modulus $\mu>0$, we know problem \eqref{eq:sto-opt} has a unique minimizer $\mb{x}^\ast$ and $\norm{\mb{x} - \mb{x}^\ast}^2 \leq \frac{2}{\mu} (f(\mb{x}) - f(\mb{x}^\ast))$ for any $\mb{x}$. (See e.g. Fact \ref{fact:sc-imply-qg} in Supplement \ref{app:basic}.) Thus, we have the following corollary for Theorem \ref{thm:PL-rate}.

\begin{Cor}\label{cor:sc-rate}
	Consider problem \eqref{eq:sto-opt} under Assumptions \ref{ass:basic} where $f(\mb{x})$ is strongly convex with modulus $\mu>0$. Under the same conditions in Theorem \ref{thm:PL-rate},  the final output $\mb{x}_t$ returned by Algorithm \ref{alg:new} satisfies 
	\begin{align}
	\mbb{E}[\norm{\mb{x}_t - \mb{x}^\ast}^2] \leq & \frac{2c_1(f(\mb{x}_{1}) - f(\mb{x}^\ast))}{\mu T^{1+\delta}} + \frac{2c_2}{\mu NT} \nonumber\\ =&  O(\frac{1}{T^{1+\delta}}) + O(\frac{1}{NT})
	\end{align}
	where  $\delta, c_1, c_2$ are positive constants defined in Theorem \ref{thm:PL-rate} and $\mb{x}^\ast$ is the unique minimizer of problem \eqref{eq:sto-opt}. 
\end{Cor}

\begin{Rem}
Recall that $O(1/T)$ convergence is optimal for stochastic strongly convex optimization \cite{book_Nemirovsky83,Rakhlin12ICML} over single node. Since the convergence of Algorithm \ref{alg:new} scales out perfectly with respect to the number of involved workers and strongly convex functions are a subclass of functions satisfying the P-L condition, we can conclude the $O(\frac{1}{NT})$ convergence attained by Algorithm \ref{alg:new} is optimal for parallel stochastic optimization under the P-L condition. It is also worth noting that we consider general stochastic optimization \eqref{eq:sto-opt} such that acceleration techniques developed for finite sum optimization, e.g., variance reduction, are excluded from consideration.
\end{Rem}

\section{General Non-Convex  Minimization}

Let $f(\mb{x})$ be the (stochastic) objective function in problem \eqref{eq:sto-opt}. For any given fixed $\mb{y}$, define a new function with respect to $\mb{x}$ given by 
\begin{align}
h_\theta(\mb{x}; \mb{y}) \defeq f(\mb{x}) + \frac{\theta}{2}\norm{\mb{x} - \mb{y}}^2 \label{eq:h-theta}
\end{align}
It is easy to verify that if $f(\mb{x})$ is smooth with modulus $L$ and $\theta > L$, then $h_\theta(\mb{x}; \mb{y})$ is both smooth with modulus $\theta+L$ and strongly convex with modulus $\theta -L > 0$. Furthermore, if $\nabla F(\mb{x}; \zeta)$ are unbiased i.i.d. stochastic gradients of function $f(\cdot)$  with a variance bounded by $\sigma^2$, then $\nabla F(\mb{x}; \zeta) + \theta(\mb{x} - \mb{y})$ are unbiased i.i.d. stochastic gradients of $h_\theta(\mb{x}; \mb{y})$ with the same variance.

\begin{algorithm}[tb]
	\caption{ $\text{CR-PSGD-Catalyst}(f, N, T, \mb{y}_0, B_1, \rho, \gamma)$}\label{alg:new-catalyst}
	\begin{algorithmic}[1]
		\STATE {\bf Input:} $N$, $T$, $\theta$, $\mb{y}_0 \in \mathbb{R}^m$, $\gamma $ , $B_1$ and  $\rho > 1$.
		\STATE Initialize $\mb{y}^{(0)} = \mb{y}_0$  and $k=1$.
		\WHILE{$k \leq \lfloor \sqrt{NT} \rfloor$}
		\STATE Define $h_\theta(\mb{x}; \mb{y}^{(k-1)})$ using \eqref{eq:h-theta}. Update $\mb{y}^{(k)}$ via
		{\scriptsize
		\begin{align*}
		\mb{y}^{(k)} = \text{CR-PSGD}(h_\theta(\cdot; \mb{y}^{(k-1)}), N, \lfloor \sqrt{T/N} \rfloor, \mb{y}^{(k-1)}, B_1, \rho, \gamma)
		\end{align*}
		}%
		\STATE Update $k \leftarrow k+1$.
		\ENDWHILE
	\end{algorithmic}
\end{algorithm}

Now consider Algorithm \ref{alg:new-catalyst} that wraps CR-PSGD with an outer-loop that updates  $h_\theta(\mb{x}; \mb{y}^{(k-1)})$ and applies CR-PSGD to minimize it. Note that $h_\theta(\mb{x}; \mb{y}^{(k-1)})$ augments the objective function $f(\mb{x})$ with an iteratively updated proximal term $\frac{\theta}{2}\norm{\mb{x} - \mb{y}^{(k-1)}}^2$. The introduction of proximal terms $\frac{\theta}{2}\norm{\mb{x} - \mb{y}^{(k-1)}}^2$ is inspired by earlier works \cite{Guler92SIOPT,He12JOTA,Salzo12JCA,Lin15NIPS,YuNeely17SIOPT, Davis17ArXiv, Paquette18AISTATS} on proximal point methods, which solve an minimization problem by solving a sequence of auxiliary problems involving a quadratic proximal term.  By choosing $\theta > L$ in \eqref{eq:h-theta}, we can ensure $h_\theta(\mb{x}; \mb{y}^{(k-1)})$ is both smooth and strongly convex.  For strongly convex $h_\theta(\mb{x}; \mb{y}^{(k-1)})$, Theorem \ref{thm:PL-rate} and Corollary \ref{cor:sc-rate} show that $\text{CR-PSGD}(N, \lfloor \sqrt{T/N} \rfloor, \mb{y}^{(k-1)}, B_1, \rho, \gamma)$ can return an $O(\frac{1}{\sqrt{NT}})$ approximated minimizer with only $O(\log(\frac{T}{N}))$ communication rounds. The ultimate goal of the proximal point like outer-loop introduced in Algorithm \ref{alg:new-catalyst} is to lift the "communication reduction" property from CR-PSGD for non-convex minimization under the restrictive PL condition to solve general non-convex minimization with reduced communication. Our method shares a similar philosophy with the ``catalyst acceleration" in \cite{Lin15NIPS} which also uses a ``proximal-point" outer-loop to achieve improved convergence rates for convex minimization by lifting fast convergence from strong convex minimization. In this perspective, we call Algorithm \ref{alg:new-catalyst}  ``CR-PSGD-Catalyst" by borrowing the word ``catalyst" from  \cite{Lin15NIPS}.  While both Algorithm \ref{alg:new-catalyst} and ``catalyst acceleration" use an proximal point outer-loop to lift desired algorithmic properties from specific problems to generic problems, they are different in the following two aspects:
\begin{itemize}
	\item The ``catalyst acceleration" in \cite{Lin15NIPS,Paquette18AISTATS} is developed to accelerate a wide range of first-order deterministic minimization, e.g., gradient based methods and their randomized variants such as SAG, SAGA, SDCA, SVRG, for both convex and non-convex cases. In particular, it requires the existence of a subprocedure with linear convergence for strongly convex minimization. It is remarked in \cite{Lin15NIPS} that  whether ``catalyst" can accelerate stochastic gradient based methods for stochastic minimization in the sense of \cite{Nemirovski09SIOPT}\footnote{For finite sum minimization, it is possible to develop linearly converging solvers by using  techniques such as variance reduction. However, for general strongly convex stochastic minimization, it is in general impossible to develop linearly converging stochastic gradient based solver and the fastest possible convergence is $O(1/T)$ \cite{Rakhlin12ICML,Hazan14JMLR, Lacoste12}. That is, stochastic minimization fundamentally fails to satisfy the prerequisite in \cite{Lin15NIPS,Paquette18AISTATS}.} remains unclear.  In contrast, our CR-PSGD-Catalyst can solve general stochastic minimization, which does not necessarily have a finite sum form, with i.i.d. stochastic gradients.  The used CR-PSGD subprocedure that is different from linear converging subprocedure used in \cite{Lin15NIPS,Paquette18AISTATS}. 
	\item The ``proximal point" outer loop used in ``catalyst acceleration" is solely to accelerate convergence \cite{Lin15NIPS,Paquette18AISTATS}. In contrast, the ``proximal point" outer loop used in our CR-PSGD-Catalyst provides convergence acceleration and communication reduction simultaneously.  Our analysis is also significantly different from analyses for conventional ``catalyst acceleration". 
\end{itemize}
Since each call of CR-PSGD in Algorithm \ref{alg:new-catalyst} requires only $O(\log(\frac{T}{N}))$ inter-worker communication rounds and there are $\sqrt{NT}$ calls of CR-PSGD,  it is easy to see CR-PSGD-Catalyst in total uses $O(\sqrt{NT}\log(\frac{T}{N}))$ communication rounds. The $O(\sqrt{NT}\log(\frac{T}{N}))$ communication complexity of CR-PSGD-Catalyst for general non-convex stochastic optimization is significantly less than the $O(T)$ communication complexity attained by PSGD \cite{Dekel12JMLR,GhadimiLan13SIOPT, Lian15NIPS} or the $O(N^{3/4}T^{3/4})$ communication complexity required by local SGD\footnote{For non-convex optimization, local SGD is more widely known as periodic model averaging or parallel restarted SGD since each worker periodically restarts its independent SGD procedure with a new initial point that is the average of all individual models \cite{Yu18ArXivAAAI,WangJoshi18ArXiv, Jiang18NIPS}.} \cite{Yu18ArXivAAAI}. The next theorem summarizes that our CR-PSGD-Catalyst can achieve the fastest known $O(1/\sqrt{NT})$ convergence that is previously attained by the PSGD or local SGD.

\begin{Thm} \label{thm:nonconvex-rate}
	Consider problem \eqref{eq:sto-opt} under Assumption \ref{ass:basic}.  If we choose $\theta > L$, $B_1 \geq 2$, $\gamma < \frac{1}{\theta + L}$ and $1<\rho < \frac{1}{1-\nu}$, where $\nu\defeq \frac{1}{2}\gamma (\theta-L)(1- (\theta +L)\gamma)$, in Algorithm \ref{alg:new-catalyst} and if $T \geq \max\{N, N\big(\frac{4c_1 (\theta+L)^2}{(\theta-L)^2}\big)^{\frac{2}{1+\delta}}, N(c_1)^{\frac{2}{1+\delta}}\}$, then we have
	{\footnotesize
	\begin{align}
	\frac{1}{\sqrt{NT}}\sum_{k=1}^{\sqrt{NT}}\mbb{E}[\norm{\nabla f(\mb{y}^{(k)})}^2] = O(\frac{1}{\sqrt{NT}}) \nonumber
	\end{align}
	}%
	where $\{\mb{y}^{(k)}, k\geq 1\}$ are a sequence of solutions returned from the CR-PSGD subprocedure.
\end{Thm}

\begin{proof}
For simplicity, we assume $\sqrt{NT}$ and $\sqrt{T/N}$ are integers and hence $\lfloor \sqrt{NT} \rfloor = \sqrt{NT}$ and $\lfloor \sqrt{T/N} \rfloor = \sqrt{T/N}$. This can be be ensured when $T = N^3 q^2$ where $q$ is any integer.  In general,  even if $\sqrt{TN}$ or $\sqrt{T/N}$ are non-integers, by using the fact that $\frac{1}{2}z \leq \lfloor z \rfloor \leq z$ for any $z\geq 2$, the same order of convergence can be easily extended to the case when $\sqrt{NT}$ or $\sqrt{T/N}$ are non-integers.

Fix $k\geq 1$ and consider stochastic minimization $\min_{\mb{x}\in \mbb{R}^m} h_\theta(\mb{x}; \mb{y}^{(k-1)})$. Since $h_\theta(\mb{x}; \mb{y}^{(k-1)})$ is strongly convex with modulus $\theta - L>0$, we know $h_\theta(\mb{x}; \mb{y}^{(k-1)})$ also satisfies the P-L condition with modulus $\theta-L$ by Fact \ref{fact:sc-imply-pl}. At the same time, $h_\theta(\mb{x}; \mb{y}^{(k-1)})$ is smooth with modulus $\theta + L$.  Note that our selections of $B_1, \gamma$ and $\rho$ satisfy the condition in Theorem \ref{thm:PL-rate} for stochastic minimization under the P-L condition. Denote $\mb{y}^{(k)}_\ast \defeq \argmin{\mb{x}\in \mbb{R}^m} \{h_\theta(\mb{x}; \mb{y}^{(k-1)})\}.$

Recall that $\mb{y}^{(k)}$ is the solution returned from CR-PSGD with $\sqrt{T/N}$ iterations. By Theorem \ref{thm:PL-rate}, we have 
{\small
\begin{align}
&\mbb{E}[h_\theta(\mb{y}^{(k)}; \mb{y}^{(k-1)}) - h_\theta(\mb{y}^{(k)}_\ast; \mb{y}^{(k-1)})] \nonumber \\\leq&  \frac{c_1}{(\frac{T}{N})^{\frac{1+\delta}{2}}} \mbb{E}[h_\theta(\mb{y}^{(k-1)}; \mb{y}^{(k-1)}) - h_\theta(\mb{y}^{(k)}_\ast; \mb{y}^{(k-1)})]+ \frac{c_2}{\sqrt{NT}} \label{eq:pf-thm-catalyst-rate-eq1}
\end{align}
}%
where  $\delta \defeq \log_{\rho}(\frac{1}{1-\nu}) -1 >0$, $c_1 \defeq  \frac{1}{1-\nu} \big(\frac{B_1}{\rho -1}\big)^{1+\delta}$, and $c_2 \defeq \frac{\rho^2\gamma(2-(\theta+L)\gamma))\sigma^2}{(1 - (1-\nu)\rho)(\rho - 1)}$ are absolute constants independent of $T$.

Since  $h_\theta(\cdot; \mb{y}^{(k-1)})$  is smooth with modulus $\theta + L$ and $\mb{y}^{(k)}_\ast$ minimizes it, by Fact \ref{fact:smooth-grad-norm-ub} (in Supplement \ref{app:basic}), we have
\begin{align}
&\frac{1}{2(\theta + L)}\norm{\nabla h_\theta(\mb{y}^{(k)}; \mb{y}^{(k-1)})}^2 \nonumber \\
\leq& h_\theta(\mb{y}^{(k)}; \mb{y}^{(k-1)}) - h_\theta(\mb{y}^{(k)}_\ast; \mb{y}^{(k-1)})\label{eq:pf-thm-catalyst-rate-eq2}
\end{align}

One the other hand ,we also have 
{\small
\begin{align}
&h_\theta(\mb{y}^{(k-1)}; \mb{y}^{(k-1)}) - h_\theta(\mb{y}^{(k)}_\ast; \mb{y}^{(k-1)}) \nonumber\\
\overset{(a)}{\leq}& \frac{\theta+L}{2}\norm{\mb{y}^{(k-1)} -\mb{y}^{(k)}_\ast }^2 \nonumber\\
\overset{(b)}{\leq} & (\theta + L)\norm{\mb{y}^{(k)} - \mb{y}^{(k)}_\ast}^2 + (\theta+L)\norm{\mb{y}^{(k)} - \mb{y}^{(k-1)}}^2 \nonumber\\
\overset{(c)}{\leq} & \frac{\theta + L}{(\theta -L)^2} \norm{\nabla h_\theta(\mb{y}^{(k)}; \mb{y}^{(k-1)})}^2+ (\theta+L)\norm{\mb{y}^{(k)} - \mb{y}^{(k-1)}}^2  \label{eq:pf-thm-catalyst-rate-eq3}
\end{align}
}%
where (a) follows from Fact \ref{fact:smooth-obj-diff-ub} (in Supplement \ref{app:basic}) by recalling again that $h_\theta(\cdot; \mb{y}^{(k-1)})$  is smooth with modulus $\theta + L$ and $\mb{y}^{(k)}_\ast$ minimizes it; (b) follows because $\norm{\mb{y}^{(k-1)} -\mb{y}^{(k)}_\ast }^2 \leq 2 \norm{\mb{y}^{(k)} - \mb{y}^{(k)}_\ast}^2 + 2\norm{\mb{y}^{(k)} - \mb{y}^{(k-1)}}^2$, which further follows by applying basic inequality $\norm{\mb{u}-\mb{v}}^2  \leq 2\norm{\mb{u}}^2 + 2\norm{\mb{v}}^2$ with $\mb{u} = \mb{y}^{(k)} - \mb{y}^{(k)}_\ast$ and $\mb{v} = \mb{y}^{(k)} - \mb{y}^{(k-1)}$; and (c) follows because $\norm{\mb{y}^{(k)} -\mb{y}^{(k)}_\ast }^2 \leq \frac{1}{(\theta -L)^2} \norm{\nabla h_\theta(\mb{y}^{(k)}; \mb{y}^{(k-1)})}^2$, which further follows from by Fact \ref{fact:sc-imply-eb} (in Supplement \ref{app:basic}) by noting that $h_\theta(\cdot; \mb{y}^{(k-1)})$  is strongly convex with modulus $\theta - L$ and $\mb{y}^{(k)}_\ast$ minimizes it.

Substituting \eqref{eq:pf-thm-catalyst-rate-eq2} and \eqref{eq:pf-thm-catalyst-rate-eq3} into \eqref{eq:pf-thm-catalyst-rate-eq1} and rearranging terms yields
{\footnotesize
\begin{align}
&\Big(\underbrace{\frac{1}{2(\theta+L)} - \frac{c_1(\theta +L)}{(\theta-L)^2} \frac{1}{(\frac{T}{N})^{\frac{1+\delta}{2}}}}_{\defeq \alpha} \Big) \mbb{E}[\norm{\nabla h_\theta(\mb{y}^{(k)}; \mb{y}^{(k-1)})}^2] \nonumber\\
\leq&  \frac{c_1(\theta +L)}{(\frac{T}{N})^{1+\delta}} \mbb{E}]\norm{\mb{y}^{(k)} - \mb{y}^{(k-1)}}^2] + \frac{c_2}{\sqrt{NT}} \label{eq:pf-thm-catalyst-rate-eq4}
\end{align}
}%

Note that $T \geq N\big(\frac{4c_1 (\theta+L)^2}{(\theta-L)^2}\big)^{\frac{2}{1+\delta}}$ ensures  the term marked by an underbrace in \eqref{eq:pf-thm-catalyst-rate-eq4} satisfies $\alpha \geq \frac{1}{4(\theta+L)}$. Thus, \eqref{eq:pf-thm-catalyst-rate-eq4} implies that
\begin{align}
&\frac{1}{4(\theta +L)} \mbb{E}[\norm{\nabla h_\theta(\mb{y}^{(k)}; \mb{y}^{(k-1)})}^2] \nonumber\\
 \leq&  \frac{c_1(\theta +L)}{(\frac{T}{N})^{1+\delta}} \mbb{E}]\norm{\mb{y}^{(k)} - \mb{y}^{(k-1)}}^2] + \frac{c_2}{\sqrt{NT}} \label{eq:pf-thm-catalyst-rate-eq5}
\end{align}

By the definition of $h_\theta(\cdot; \mb{y}^{(k-1)})$, we have $\nabla h_\theta(\mb{y}^{(k)}; \mb{y}^{(k-1)}) = \nabla f(\mb{y}^{(k)}) + \theta  (\mb{y}^{(k)} - \mb{y}^{(k-1)})$. This implies that  
{\small
\begin{align}
\norm{ \nabla f(\mb{y}^{(k)})}^2 \leq &2 \norm{\nabla h_\theta(\mb{y}^{(k)}; \mb{y}^{(k-1)}) }^2 + 2\theta^2 \norm{\mb{y}^{(k)} - \mb{y}^{(k-1)}}^2 \label{eq:pf-thm-catalyst-rate-eq6}
\end{align}
}%
Combining  \eqref{eq:pf-thm-catalyst-rate-eq5} and \eqref{eq:pf-thm-catalyst-rate-eq6} yields
{\scriptsize
\begin{align}
&\mbb{E}[\norm{\nabla f(\mb{y}^{(k)})}^2]  \nonumber \\
\leq& \Big(\frac{8c_1 (\theta +L)^2}{(\frac{T}{N})^{1+\delta}}+ 2\theta^2 \Big) \mbb{E}[\norm{\mb{y}^{(k)} - \mb{y}^{(k-1)}}^2] + \frac{8c_2(\theta +L)}{\sqrt{NT}} \nonumber \\
\overset{(a)}{\leq}&  \Big( 8c_1 (\theta +L)^2+ 2\theta^2 \Big) \mbb{E}[\norm{\mb{y}^{(k)} - \mb{y}^{(k-1)}}^2] + \frac{8c_2(\theta +L)}{ \sqrt{NT}}\label{eq:pf-thm-catalyst-rate-eq7}
\end{align}
}%
where (a) follows because $(\frac{T}{N})^{1+\delta} \geq 1$ as long as $T\geq N$.
 
Since $T \geq Nc_1^{\frac{2}{1+\delta}}$ ensures $ \frac{c_1}{(\frac{T}{N})^{\frac{1+\delta}{2}}} \leq 1$, by \eqref{eq:pf-thm-catalyst-rate-eq1}, we have
 \begin{align}
 &\mbb{E}[h_\theta(\mb{y}^{(k)}; \mb{y}^{(k-1)}) - h_\theta(\mb{y}^{(k)}_\ast; \mb{y}^{(k-1)})] \nonumber \\
 \leq&  \mbb{E}[h_\theta(\mb{y}^{(k-1)}; \mb{y}^{(k-1)}) - h_\theta(\mb{y}^{(k)}_\ast; \mb{y}^{(k-1)})]+ \frac{c_2}{\sqrt{NT}} \label{eq:pf-thm-catalyst-rate-eq8}
 \end{align}
Cancelling the common term on both sides and substituting the definition of $h_\theta (\cdot; \mb{y}^{(k-1)})$ into \eqref{eq:pf-thm-catalyst-rate-eq8} yields

\begin{align}
&\mbb{E}[f(\mb{y}^{(k)}) + \frac{\theta}{2} \norm{\mb{y}^{(k)} - \mb{y}^{(k-1)}}^2] \nonumber \\
 \leq& \mbb{E}[f(\mb{y}^{(k-1)})] + \frac{c_2}{\sqrt{NT}}
\end{align}
Rewriting this inequality as $\mbb{E}[\norm{\mb{y}^{(k)} - \mb{y}^{(k-1)}}^2] \leq \frac{2}{\theta} \mbb{E}[f(\mb{y}^{(k-1)}) - f(\mb{y}^{(k)})] + \frac{2c_2}{\theta \sqrt{NT}}$ and substituting it into \eqref{eq:pf-thm-catalyst-rate-eq7} yields
{\small
\begin{align}
&\mbb{E}[\norm{\nabla f(\mb{y}^{(k)})}^2] \nonumber\\
\leq& \frac{2}{\theta}\Big(8c_1 (\theta +L)^2+ 2\theta^2 \Big)\mbb{E}[f(\mb{y}^{(k-1)}) - f(\mb{y}^{(k)})] \nonumber \\&+ \Big(\frac{16 c_1 (\theta +L)^2}{\theta}+ 12\theta + 8L\Big) \frac{c_2}{\sqrt{NT}}
\end{align}
}%

Summing this inequality over $k\in\{1, \ldots, \sqrt{NT}\}$ and dividing both sides by a factor $\sqrt{NT}$ yields
{\small
\begin{align}
&\frac{1}{\sqrt{NT}}\sum_{k=1}^{\sqrt{NT}}\mbb{E}[\norm{\nabla f(\mb{y}^{(k)})}^2] \nonumber\\
\leq& \frac{2}{\theta}\Big(8 c_1 (\theta +L)^2+ 2\theta^2 \Big)\frac{\mbb{E}[f(\mb{y}^{(0)}) - f(\mb{y}^{\sqrt{NT}})]}{\sqrt{NT}} \nonumber \\ &+ \Big(\frac{16c_1 (\theta +L)^2}{\theta }+ 12\theta+8L\Big) \frac{c_2}{\sqrt{NT}} \nonumber\\
\overset{(a)}{\leq} & \frac{2}{\theta}\Big(8c_1 (\theta +L)^2+ 2\theta^2 \Big)\frac{f(\mb{y}^{(0)}) - f^\ast}{\sqrt{NT}} \nonumber\\ &+ \Big(\frac{16c_1 (\theta +L)^2}{\theta}+ 12\theta+ 8L\Big) \frac{c_2}{\sqrt{NT}} \nonumber\\
= &O(\frac{1}{\sqrt{NT}})
\end{align}
}%
where (a) follows because $f^\ast$ is the global minimum of problem \eqref{eq:sto-opt}.

\end{proof}

\section{Experiments}
To validate the theory developed in this paper, we conduct two numerical experiments: (1) distributed logistic regression and (2) training deep neural networks.

\subsection{Distributed Logistic Regression}

Consider solving an $l_2$ regularized logistic regression problem using multiple parallel nodes. Let $(\mb{z}_{ij}, b_{ij})$ be the training pairs at node $i$, where$\mathbf{z}_{ij} \in  \mathbb{R}^{d}$ are $d$-dimension feature vectors and $b_{ij} \in \{-1,1\}$ are labels.  The problem can be cast as follows:
\begin{align}
\min_{\mb{x}\in \mbb{R}^d} &  ~ \frac{1}{N}\sum_{i=1}^{N} \frac{1}{M_{i}}\sum_{j=1}^{M_{i}}  \log(1+ \exp(b_{ij} (\mathbf{z}_{ij}\tran \mathbf{x}_{i})) + \frac{1}{2}\mu \norm{\mathbf{x}}^2 \label{eq:logstic-reg}
\end{align}
where $N$ is the number of parallel workers, $M_i$ are the number of training samples available at node $i$ and $\mu$ is the regularization coefficient.

Our experiment generates a problem instance with  $d=500$, $N=10$, $M_{i} = 10^{4}, \forall i\in\{1,2,\ldots,N\}$ and $\mu = 0.001$. The synthetic training feature vectors $\mathbf{z}_{ij}$ are generated from normal distribution $\mathcal{N}(\mb{I}, 4\mb{I}_d)$. Assume the underlying classification problem has a true weight vector $\mathbf{x}^{\text{true}}\in \mathbb{R}^{d}$ generated from a standard normal distribution and then generate the noisy labels $b_{ij} = \text{sign}(\mathbf{z}_{ij}\tran \mathbf{x}^{\text{true}}  + \xi_{i})$ where noise $\xi_{i}\sim \mathcal{N}(0, 1)$.  Note that the distributed logistic regression problem \eqref{eq:logstic-reg}  is strongly convex and hence satisfies Assumption \ref{ass:P-L}.   We run Algorithm \ref{alg:new-catalyst},  the classical parallel SGD, and ``local SGD" with communication skipping proposed in \cite{Stich18ArXiv} to solve problem \eqref{eq:logstic-reg}. For strongly  convex stochastic optimization, all these three methods are proven to achieve the fast $O(\frac{1}{NT})$ convergence. The communication complexity of these three methods are $O(\log(T))$, $O(T)$ and $O(\sqrt{NT})$, respectively. Our Algorithm \ref{alg:new} has the lowest communication complexity. In the experiment,  we choose $N=10$, $T=10000$, $\mb{x}_1 = \mb{0}$, $B_1 =2$, $\gamma = 0.1$ and $\rho = 1.1$ in Algorithm \ref{alg:new}; choose fixed batch size $2$ and learning rate $0.1$ in the classical parallel SGD; choose fixed batch size $2$,  learning rate $0.1$ and the largest communication skipping interval for which the loss at convergence does not sacrifice in local SGD. Figures \ref{fig:sc_sfo} and \ref{fig:sc_commu} plot the objective values of problem \eqref{eq:logstic-reg} versus the number of SFO access and the number of communication rounds, respectively.   Our numerical results verify that Algorithm \ref{alg:new} can achieve similar convergence as existing fastest parallel SGD variants with fewer communication rounds.

\begin{figure}[h!] 
\centering
\includegraphics[width=0.48\textwidth]{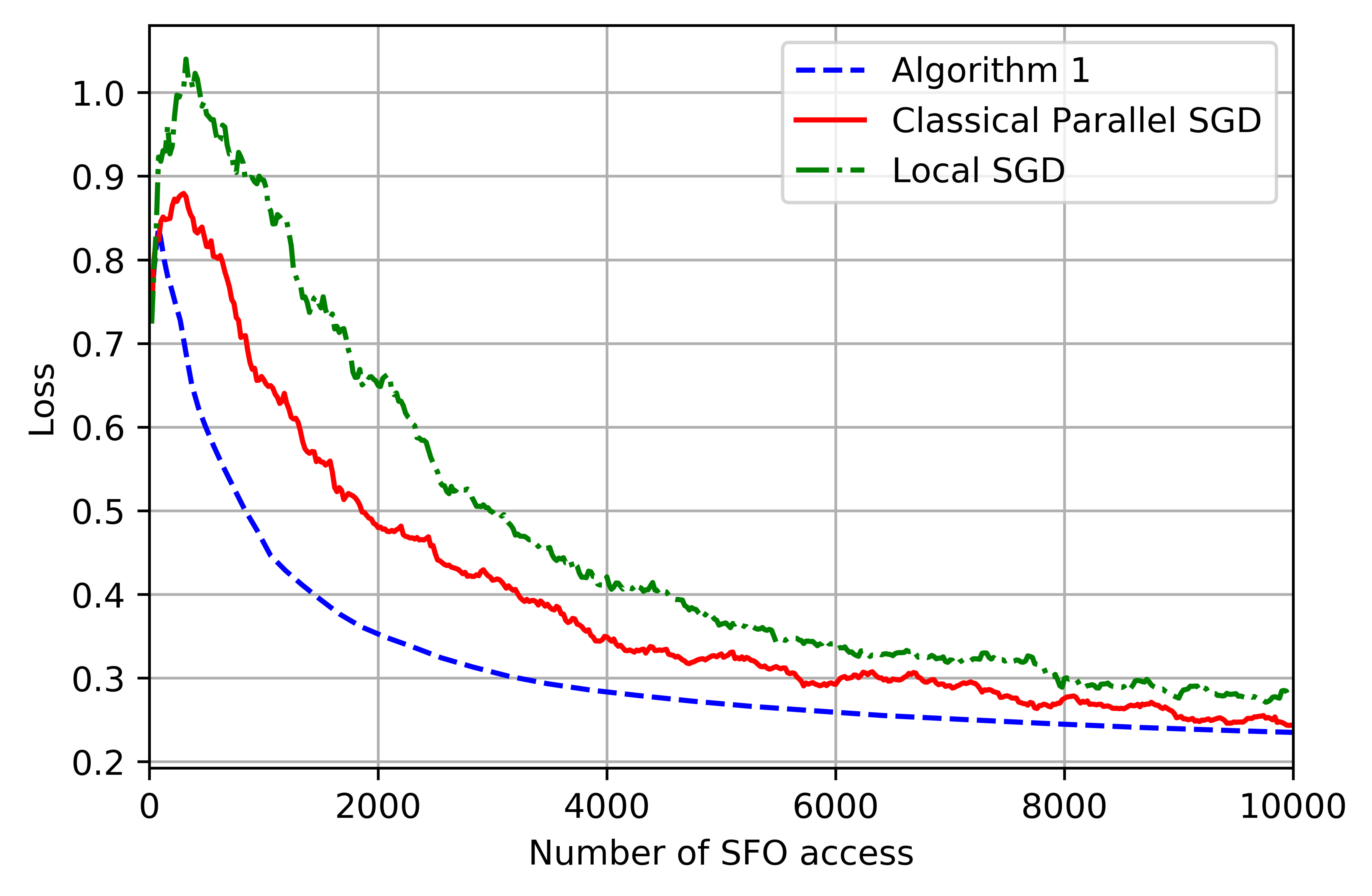}
\caption{Distributed logistic regression: loss v.s. number of SFO access.}\label{fig:sc_sfo}
\end{figure}
\begin{figure}[h!] 
\centering
\includegraphics[width=0.48\textwidth]{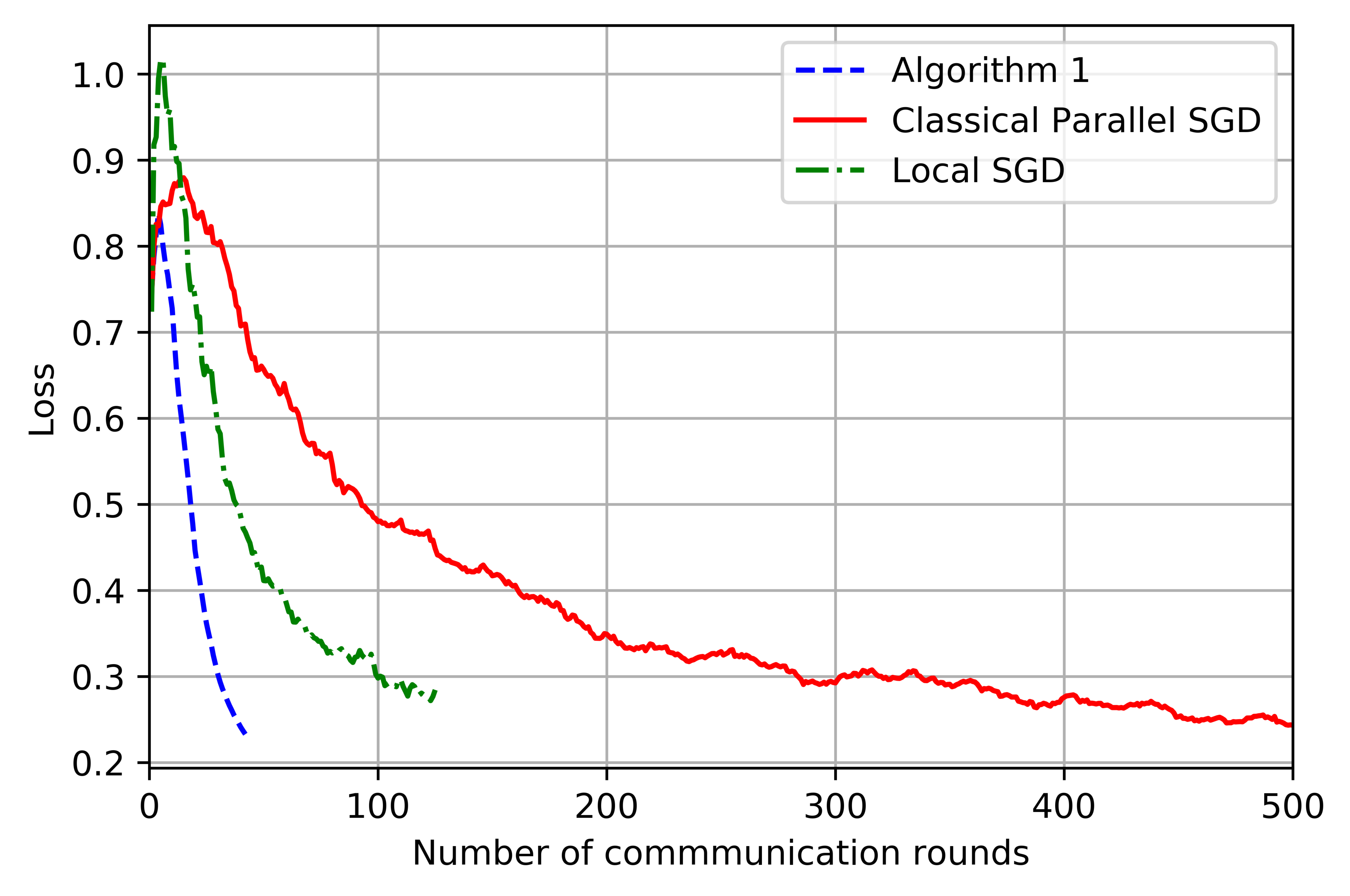}
\caption{Distributed logistic regression: loss v.s. number of communication rounds.}\label{fig:sc_commu}
\end{figure}

\subsection{Training Deep Neural Networks}
Consider using deep learning for the image classification over CIFAR-10 \cite{Krizhevsky09}. The loss function for deep neural networks is non-convex and typically violates Assumption \ref{ass:P-L}. We run Algorithm \ref{alg:new-catalyst},  the classical parallel SGD, and ``local SGD" with communication skipping in \cite{Stich18ArXiv, Yu18ArXivAAAI} to train ResNet20 \cite{He16CVPR} with $8$ GPUs.  It has been shown that the ``local SGD", also known as parallel restarted SGD or periodic model averaging, can linearly speed up the parallel training of deep neural networks with significantly less communication overhead than the classical parallel SGD \cite{Yu18ArXivAAAI,Lin18ArXiv,WangJoshi18ArXiv, Jiang18NIPS}.  For both parallel SGD and local SGD, the learning rate is $0.1$, the momentum is $0.9$, the weight decay is $1e-4$, and the batch size at each GPU is $32$.  For local SGD, we use the largest communication skipping interval for which the loss at convergence does not sacrifice.  For Algorithm \ref{alg:new-catalyst}, we use $B_1=32$, $\rho=1.02$ and $\gamma =0.1$. In our experiment, each iteration of Algorithm \ref{alg:new-catalyst} executes CR-PSGD (Algorithm \ref{alg:new}) to access one epoch of training data at each GPU. That is, the $T$ parameter in each call of Algorithm \ref{alg:new} is $50000$. The $B_\tau$ parameter in Algorithm \ref{alg:new} stop growing when it exceeds $512$.  

\begin{figure}[h!] 
\centering
\includegraphics[width=0.48\textwidth]{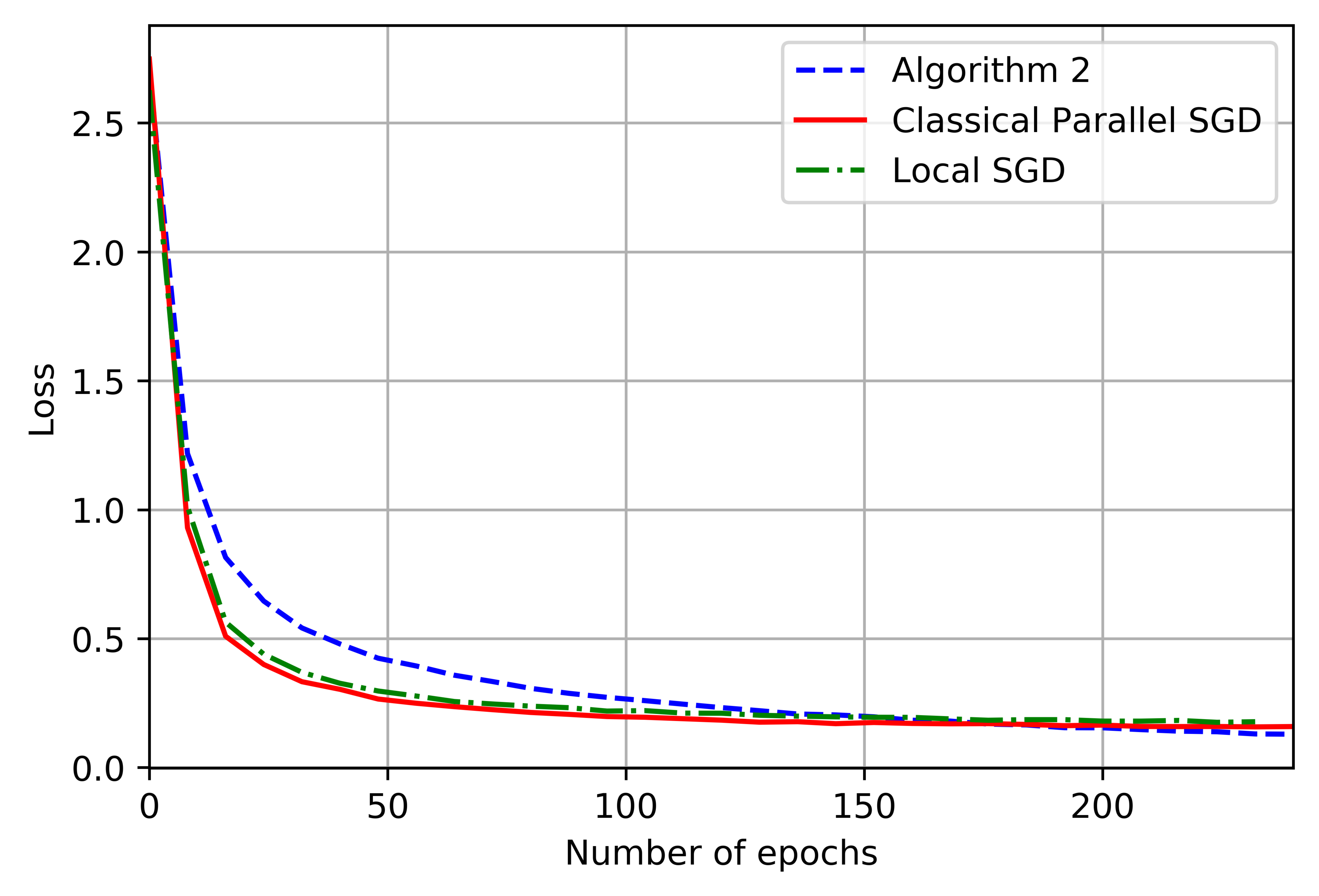}
\caption{Training deep neural networks: loss v.s. number of SFO access.}\label{fig:dl_sfo}
\end{figure}
\begin{figure}[h!] 
\centering
\includegraphics[width=0.48\textwidth]{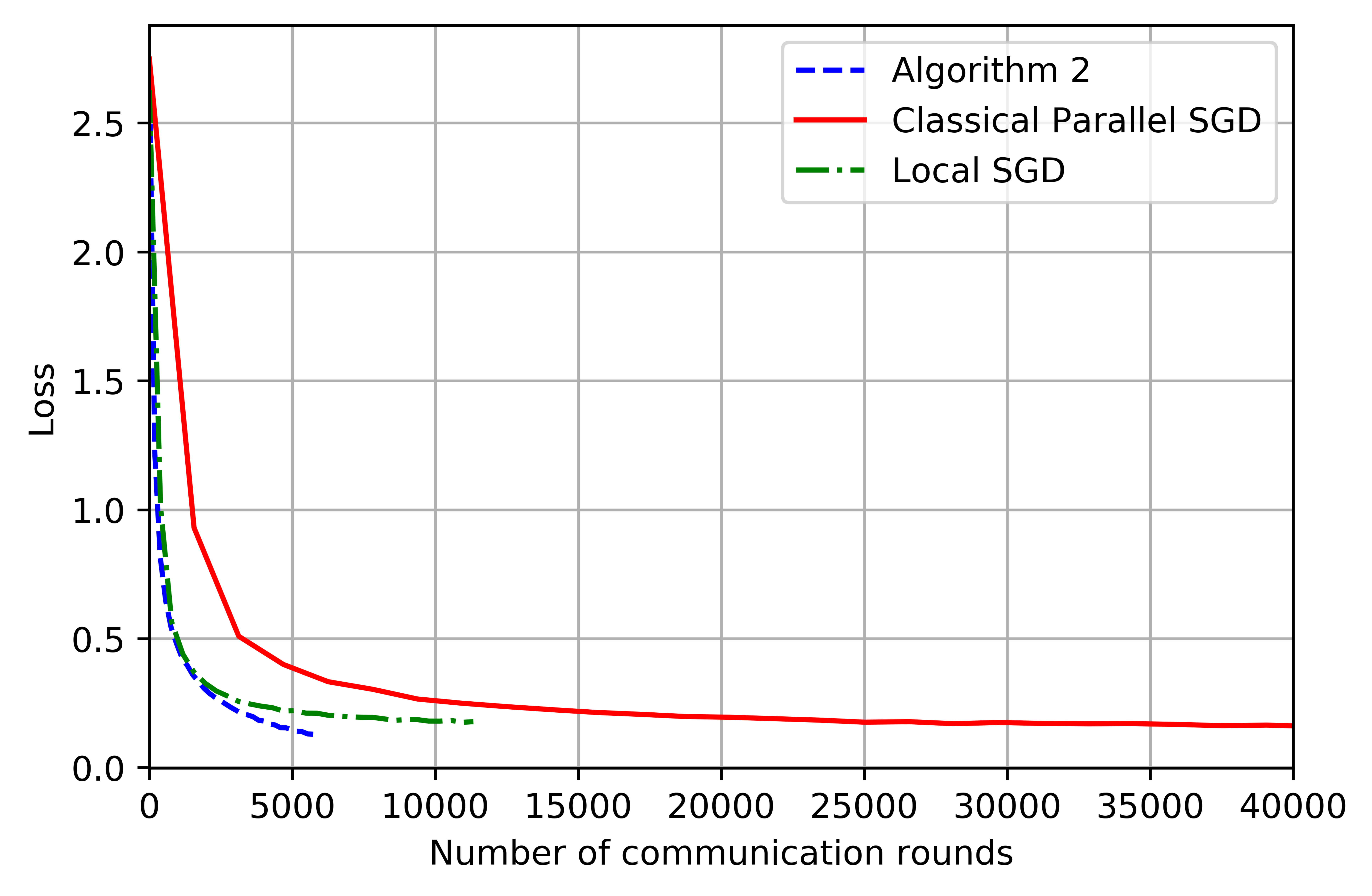}
\caption{Training deep neural networks: loss v.s. number of communication rounds.}\label{fig:dl_commu}
\end{figure}

\vspace{-1em}
\section{Conclusion}
In this paper, we explore the idea of using dynamic batch sizes for distributed non-convex optimization.  For non-convex optimization satisfying the Polyak-Lojasiewicz (P-L) condition, we show using exponential increasing batch sizes in parallel SGD as in Algorithm \ref{alg:new} can achieve $O(\frac{1}{NT})$ convergence using only $O(\log(T))$ communication rounds. For general stochastic non-convex optimization (without P-L condition), we propose a Catalyst-like algorithm that can achieve $O(\frac{1}{\sqrt{NT}})$ convergence with $O(\sqrt{TN}\log(\frac{T}{N}))$ communication rounds.

\bibliographystyle{icml2019} 
\bibliography{mybibfile}

\begin{thebibliography}{35}
\providecommand{\natexlab}[1]{#1}
\providecommand{\url}[1]{\texttt{#1}}
\expandafter\ifx\csname urlstyle\endcsname\relax
  \providecommand{\doi}[1]{doi: #1}\else
  \providecommand{\doi}{doi: \begingroup \urlstyle{rm}\Url}\fi

\bibitem[Bertsekas(1999)]{book_NonlinearProgramming_Bertsekas}
Bertsekas, D.~P.
\newblock \emph{Nonlinear Programming}.
\newblock Athena Scientific, second edition, 1999.

\bibitem[Bottou \& Bousquet(2008)Bottou and Bousquet]{Bottou08NIPS}
Bottou, L. and Bousquet, O.
\newblock The tradeoffs of large scale learning.
\newblock In \emph{Advances in Neural Information Processing Systems (NIPS)},
  2008.

\bibitem[Bottou et~al.(2018)Bottou, Curtis, and Nocedal]{Bottou18SIAM}
Bottou, L., Curtis, F.~E., and Nocedal, J.
\newblock Optimization methods for large-scale machine learning.
\newblock \emph{SIAM Review}, 60\penalty0 (2):\penalty0 223--311, 2018.

\bibitem[Davis \& Grimmer(2017)Davis and Grimmer]{Davis17ArXiv}
Davis, D. and Grimmer, B.
\newblock Proximally guided stochastic subgradient method for nonsmooth,
  nonconvex problems.
\newblock \emph{arXiv:1707.03505}, 2017.

\bibitem[De et~al.(2017)De, Yadav, Jacobs, and Goldstein]{DE17AISTATS}
De, S., Yadav, A., Jacobs, D., and Goldstein, T.
\newblock Automated inference with adaptive batches.
\newblock In \emph{International Conference on Artificial Intelligence and
  Statistics (AISTATS)}, pp.\  1504--1513, 2017.

\bibitem[Dekel et~al.(2012)Dekel, Gilad-Bachrach, Shamir, and
  Xiao]{Dekel12JMLR}
Dekel, O., Gilad-Bachrach, R., Shamir, O., and Xiao, L.
\newblock Optimal distributed online prediction using mini-batches.
\newblock \emph{Journal of Machine Learning Research}, 13\penalty0 (165--202),
  2012.

\bibitem[Devarakonda et~al.(2017)Devarakonda, Naumov, and
  Garland]{Devarakonda17ArXiv}
Devarakonda, A., Naumov, M., and Garland, M.
\newblock Adabatch: Adaptive batch sizes for training deep neural networks.
\newblock \emph{arXiv:1712.02029}, 2017.

\bibitem[Friedlander \& Schmidt(2012)Friedlander and
  Schmidt]{Friedlander12SIOPT}
Friedlander, M.~P. and Schmidt, M.
\newblock Hybrid deterministic-stochastic methods for data fitting.
\newblock \emph{SIAM Journal on Scientific Computing}, 34\penalty0
  (3):\penalty0 1380--1405, 2012.

\bibitem[Ghadimi \& Lan(2013)Ghadimi and Lan]{GhadimiLan13SIOPT}
Ghadimi, S. and Lan, G.
\newblock Stochastic first-and zeroth-order methods for nonconvex stochastic
  programming.
\newblock \emph{SIAM Journal on Optimization}, 23\penalty0 (4):\penalty0
  2341--2368, 2013.

\bibitem[Ghadimi et~al.(2016)Ghadimi, Lan, and Zhang]{Ghadimi16MP}
Ghadimi, S., Lan, G., and Zhang, H.
\newblock Mini-batch stochastic approximation methods for nonconvex stochastic
  composite optimization.
\newblock \emph{Mathematical Programming}, 155\penalty0 (1-2):\penalty0
  267--305, 2016.

\bibitem[G{\"u}ler(1992)]{Guler92SIOPT}
G{\"u}ler, O.
\newblock New proximal point algorithms for convex minimization.
\newblock \emph{SIAM Journal on Optimization}, 2\penalty0 (4):\penalty0
  649--664, 1992.

\bibitem[Hazan \& Kale(2014)Hazan and Kale]{Hazan14JMLR}
Hazan, E. and Kale, S.
\newblock Beyond the regret minimization barrier: an optimal algorithm for
  stochastic strongly-convex optimization.
\newblock \emph{Journal of Machine Learning Research}, 2014.

\bibitem[He \& Yuan(2012)He and Yuan]{He12JOTA}
He, B. and Yuan, X.
\newblock An accelerated inexact proximal point algorithm for convex
  minimization.
\newblock \emph{Journal of Optimization Theory and Applications}, 154\penalty0
  (2):\penalty0 536--548, 2012.

\bibitem[He et~al.(2016)He, Zhang, Ren, and Sun]{He16CVPR}
He, K., Zhang, X., Ren, S., and Sun, J.
\newblock Deep residual learning for image recognition.
\newblock In \emph{IEEE conference on computer vision and pattern recognition
  (CVPR)}, 2016.

\bibitem[Jiang \& Agrawal(2018)Jiang and Agrawal]{Jiang18NIPS}
Jiang, P. and Agrawal, G.
\newblock A linear speedup analysis of distributed deep learning with sparse
  and quantized communication.
\newblock In \emph{Advances in Neural Information Processing Systems
  (NeurIPS)}, 2018.

\bibitem[Karimi et~al.(2016)Karimi, Nutini, and Schmidt]{Karimi16}
Karimi, H., Nutini, J., and Schmidt, M.
\newblock Linear convergence of gradient and proximal-gradient methods under
  the {P}olyak-{L}ojasiewicz condition.
\newblock In \emph{Joint European Conference on Machine Learning and Knowledge
  Discovery in Databases}, 2016.

\bibitem[Krizhevsky \& Hinton(2009)Krizhevsky and Hinton]{Krizhevsky09}
Krizhevsky, A. and Hinton, G.
\newblock Learning multiple layers of features from tiny images.
\newblock \emph{Technical report, University of Toronto}, 2009.

\bibitem[Lacoste-Julien et~al.(2012)Lacoste-Julien, Schmidt, and
  Bach]{Lacoste12}
Lacoste-Julien, S., Schmidt, M., and Bach, F.
\newblock A simpler approach to obtaining an ${O}(1/t)$ convergence rate for
  the projected stochastic subgradient method.
\newblock \emph{arXiv:1212.2002}, 2012.

\bibitem[Lian et~al.(2015)Lian, Huang, Li, and Liu]{Lian15NIPS}
Lian, X., Huang, Y., Li, Y., and Liu, J.
\newblock Asynchronous parallel stochastic gradient for nonconvex optimization.
\newblock In \emph{Advances in Neural Information Processing Systems (NIPS)},
  2015.

\bibitem[Lin et~al.(2015)Lin, Mairal, and Harchaoui]{Lin15NIPS}
Lin, H., Mairal, J., and Harchaoui, Z.
\newblock A universal catalyst for first-order optimization.
\newblock In \emph{Advances in Neural Information Processing Systems (NIPS)},
  pp.\  3384--3392, 2015.

\bibitem[Lin et~al.(2018)Lin, Stich, and Jaggi]{Lin18ArXiv}
Lin, T., Stich, S.~U., and Jaggi, M.
\newblock Don't use large mini-batches, use local {SGD}.
\newblock \emph{arXiv:1808.07217}, 2018.

\bibitem[Nemirovski et~al.(2009)Nemirovski, Juditsky, Lan, and
  Shapiro]{Nemirovski09SIOPT}
Nemirovski, A., Juditsky, A., Lan, G., and Shapiro, A.
\newblock Robust stochastic approximation approach to stochastic programming.
\newblock \emph{SIAM Journal on optimization}, 19\penalty0 (4):\penalty0
  1574--1609, 2009.

\bibitem[Nemirovsky \& Yudin(1983)Nemirovsky and Yudin]{book_Nemirovsky83}
Nemirovsky, A.~S. and Yudin, D.~B.
\newblock \emph{Problem complexity and method efficiency in optimization.}
\newblock 1983.

\bibitem[Nesterov(2004)]{book_ConvexOpt_Nesterov}
Nesterov, Y.
\newblock \emph{Introductory Lectures on Convex Optimization: A Basic Course}.
\newblock Springer Science \& Business Media, 2004.

\bibitem[Paquette et~al.(2018)Paquette, Lin, Drusvyatskiy, Mairal, and
  Harchaoui]{Paquette18AISTATS}
Paquette, C., Lin, H., Drusvyatskiy, D., Mairal, J., and Harchaoui, Z.
\newblock Catalyst for gradient-based nonconvex optimization.
\newblock In \emph{International Conference on Artificial Intelligence and
  Statistics (AISTATS)}, pp.\  1--10, 2018.

\bibitem[Polyak(1963)]{Polyak63}
Polyak, B.~T.
\newblock Gradient methods for minimizing functionals.
\newblock \emph{Zhurnal Vychislitel'noi Matematikii Matematicheskoi Fiziki},
  pp.\  643--653, 1963.

\bibitem[Rakhlin et~al.(2012)Rakhlin, Shamir, and Sridharan]{Rakhlin12ICML}
Rakhlin, A., Shamir, O., and Sridharan, K.
\newblock Making gradient descent optimal for strongly convex stochastic
  optimization.
\newblock In \emph{Proceedings of International Conference on Machine Learning
  (ICML)}, 2012.

\bibitem[Salzo \& Villa(2012)Salzo and Villa]{Salzo12JCA}
Salzo, S. and Villa, S.
\newblock Inexact and accelerated proximal point algorithms.
\newblock \emph{Journal of Convex Analysis}, 19\penalty0 (4):\penalty0
  1167--1192, 2012.

\bibitem[Smith \& Le(2018)Smith and Le]{SmithLe18ICLR}
Smith, S.~L. and Le, Q.~V.
\newblock Understanding generalization and stochastic gradient descent.
\newblock In \emph{Proceedings of the International Conference on Learning
  Representations (ICLR)}, 2018.

\bibitem[Smith et~al.(2018)Smith, Kindermans, Ying, and Le]{Smith18ICLR}
Smith, S.~L., Kindermans, P.-J., Ying, C., and Le, Q.~V.
\newblock Don't decay the learning rate, increase the batch size.
\newblock In \emph{Proceedings of the International Conference on Learning
  Representations (ICLR)}, 2018.

\bibitem[Stich(2018)]{Stich18ArXiv}
Stich, S.~U.
\newblock Local {SGD} converges fast and communicates little.
\newblock \emph{arXiv:1805.09767}, 2018.

\bibitem[Wang \& Joshi(2018)Wang and Joshi]{WangJoshi18ArXiv}
Wang, J. and Joshi, G.
\newblock Cooperative {SGD}: A unified framework for the design and analysis of
  communication-efficient {SGD} algorithms.
\newblock \emph{arXiv:1808.07576}, 2018.

\bibitem[Yu \& Neely(2017)Yu and Neely]{YuNeely17SIOPT}
Yu, H. and Neely, M.~J.
\newblock A simple parallel algorithm with an ${O}(1/t)$ convergence rate for
  general convex programs.
\newblock \emph{SIAM Journal on Optimization}, 27\penalty0 (2):\penalty0
  759--783, 2017.

\bibitem[Yu et~al.(2018)Yu, Yang, and Zhu]{Yu18ArXivAAAI}
Yu, H., Yang, S., and Zhu, S.
\newblock Parallel restarted {SGD} with faster convergence and less
  communication: Demystifying why model averaging works for deep learning.
\newblock \emph{arXiv:1807.06629}, 2018.

\bibitem[Zhang et~al.(2013)Zhang, Yang, Jin, and He]{Zhang13ICML}
Zhang, L., Yang, T., Jin, R., and He, X.
\newblock ${O}(\log{T})$ projections for stochastic optimization of smooth and
  strongly convex functions.
\newblock In \emph{International Conference on Machine Learning (ICML)}, pp.\
  1121--1129, 2013.

\end{thebibliography}

\onecolumn
\section{Supplement}
\subsection{Basic Facts} \label{app:basic}

This section summarizes several well-known facts for smooth and/or strongly convex functions. For the convenience to the readers, we also provide  self-contained proofs to these facts. 

Recall that if $\phi(\mb{x})$ is a smooth function with modulus $L>0$, then we have $\phi(\mb{y}) \leq \phi(\mb{x}) + \inprod{\nabla \phi(\mb{x}), \mb{y}-\mb{x}} + \frac{L}{2}\norm{\mb{y}-\mb{x}}^2$ for any $\mb{x}$ and $ \mb{y}$. This property is known as the descent lemma for smooth functions, see e.g., Proposition A.24 in \cite{book_NonlinearProgramming_Bertsekas}.  The next two useful facts follow directly from the descent lemma.

\begin{Fact}\label{fact:smooth-obj-diff-ub}
	Let $\phi:\mbb{R}^m \mapsto \mbb{R}$ be a smooth function with modulus $L$. If $\mb{x}^\ast$ is a global minimizer of $f$ over $\mbb{R}^m$, then 
	\begin{align}
	\phi(\mb{x}) - \phi(\mb{x}^\ast) \leq \frac{L}{2}\norm{\mb{x} - \mb{x}^\ast}^2, \quad \forall \mb{x}
	\end{align}
\end{Fact}
\begin{proof}
	By the descent lemma for smooth functions, for any $\mb{x}$, we have 
	\begin{align*}
	\phi(\mb{x})  \leq& \phi(\mb{x}^\ast) + \inprod{\nabla \phi(\mb{x}^\ast), \mb{x} - \mb{x}^\ast} + \frac{L}{2}\norm{\mb{x} - \mb{x}^\ast}^2 \\
	\overset{(a)}{=} &\phi(\mb{x}^\ast) + \frac{L}{2}\norm{\mb{x} - \mb{x}^\ast}^2 
	\end{align*}
	where (a) follows from $\nabla \phi(\mb{x}^\ast) = \mb{0}$.
\end{proof}

\begin{Fact}\label{fact:smooth-grad-norm-ub}
	Let $\phi: \mbb{R}^m \rightarrow \mbb{R}$ be a smooth function with modulus $L$. We have 
	\begin{align}
	\frac{1}{2L}\norm{\nabla \phi(\mb{x})}^2 \leq \phi(\mb{x}) - \phi^\ast, \quad \forall \mb{x}
	\end{align}
	where $\phi^\ast$ is the global minimum of $\phi(\mb{x})$.
\end{Fact}
\begin{proof}
	By the descent lemma for smooth functions, for any $\mb{x}, \mb{y}\in \mbb{R}^n$, we have 
	\begin{align*}
	\phi(\mb{y}) \leq& \phi(\mb{x}) + \inprod{\nabla \phi(\mb{x}), \mb{y} - \mb{x}} + \frac{L}{2} \norm{\mb{y}-\mb{x}}^2\\
	\overset{(a)}{=}& \phi(\mb{x}) + \frac{L}{2} \norm{\mb{y}-\mb{x} + \frac{1}{L}\nabla \phi(\mb{x})}^2 - \frac{1}{2L}\norm{\nabla \phi(\mb{x})}^2
	\end{align*}
	where (a) can be verified by noting that $ \norm{\mb{y}-\mb{x} + \frac{1}{L}\nabla f(\mb{x})}^2 = \norm{\mb{y}-\mb{x} }^2 + \frac{2}{L}\inprod{\nabla \phi(\mb{x}), \mb{y}-\mb{x}}+ \frac{1}{L^2}\norm{\nabla \phi(\mb{x})}^2$.
	
	Minimizing both sides over $\mb{y}\in \mbb{R}^m $ yields
	\begin{align*}
	\phi^\ast \leq \phi(\mb{x})- \frac{1}{2L}\norm{\nabla \phi(\mb{x})}^2  
	\end{align*}
\end{proof}

Recall that if smooth function $\phi(\mb{x})$ is strongly convex with modulus $\mu>0$, then we have $\phi(\mb{y}) \geq \phi(\mb{x}) + \inprod{\nabla \phi(\mb{x}), \mb{y}-\mb{x}} + \frac{\mu}{2}\norm{\mb{y}-\mb{x}}^2$ for any $\mb{x}$ and $ \mb{y}$. The next two useful facts follow directly from this inequality.

\begin{Fact}\label{fact:sc-imply-qg}
	Let smooth function $\phi:\mbb{R}^m \mapsto \mbb{R}$ be strongly convex with modulus $\mu>0$. If $\mb{x}^\ast$ is the (unique) global minimizer of $f$ over $\mbb{R}^m$, then 
	\begin{align}
	\phi(\mb{x}) - \phi(\mb{x}^\ast) \geq \frac{\mu}{2}\norm{\mb{x} - \mb{x}^\ast}^2, \quad \forall \mb{x} \label{eq:sc-quadratic-growth}
	\end{align}
\end{Fact}
\begin{proof}
	By the strong convexity of $\phi(\mb{x})$, for any $\mb{x}$, we have 
	\begin{align*}
	\phi(\mb{x})  \geq& \phi(\mb{x}^\ast) + \inprod{\nabla \phi(\mb{x}^\ast), \mb{x} - \mb{x}^\ast} + \frac{\mu}{2}\norm{\mb{x} - \mb{x}^\ast}^2 \\
	\overset{(a)}{=} &\phi(\mb{x}^\ast) + \frac{\mu}{2}\norm{\mb{x} - \mb{x}^\ast}^2 
	\end{align*}
	where (a) follows from $\nabla \phi(\mb{x}^\ast) = \mb{0}$.
\end{proof}

\begin{Fact}\label{fact:sc-imply-eb}
	Let smooth function $\phi: \mbb{R}^m \rightarrow \mbb{R}$  be strongly convex with modulus $\mu>0$. If $\mb{x}^\ast$ is the (unique) global minimizer of $\phi(\mb{x})$ over $\mbb{R}^m$, then 
	\begin{align}
	\norm{\nabla \phi(\mb{x})} \geq \mu \norm{\mb{x} - \mb{x}^\ast}, \quad \forall \mb{x} \label{eq:sc-error-bound}
	\end{align}
\end{Fact}
\begin{proof}
	By Fact \ref{fact:sc-imply-qg}, we have
	\begin{align*}
	\phi(\mb{x}) - \phi(\mb{x}^\ast) \geq \frac{\mu}{2}\norm{\mb{x} - \mb{x}^\ast}^2, \quad \forall \mb{x}
	\end{align*} 
	By Fact \ref{fact:sc-imply-pl} and the definition of P-L condition, we have 
	\begin{align*}
	\frac{1}{2}\norm{\nabla \phi(\mb{x})}^2 \geq \mu (\phi(\mb{x}) -  \phi(\mb{x}^\ast), \quad \forall \mb{x}
	\end{align*}
	Combining these two inequalities yields the desired result.
\end{proof}

Both Fact \ref{fact:sc-imply-qg} and Fact \ref{fact:sc-imply-eb} are restricted to strongly convex functions. They can be possibly extended to smooth functions without strong convexity. A generalization of \eqref{eq:sc-quadratic-growth} is known as the quadratic growth condition. Similarly, a generalization of \eqref{eq:sc-error-bound} is known as the error bound condition. In general, both \eqref{eq:sc-quadratic-growth} and \eqref{eq:sc-error-bound}, where $\mb{x}^\ast$ should be replaced by $\mc{P}_{\mathcal{X}^\ast}[\mb{x}]$, i.e., the projection of $\mb{x}$ onto the set of minimizers for $\phi(\mb{x})$ when $\phi(\mb{x})$ does not have a unique minimizer, can be proven to hold as long as smooth $\phi(\mb{x})$ satisfies the P-L condition with the same modulus $\mu$. See Supplement A in \cite{Karimi16} for detailed discussions.

\subsection{Proof of Theorem \ref{thm:PL-rate}} \label{sec:pf-PL-rate}

Fix $T>1$. Let $\mb{x}_t$ be the solution returned by Algorithm \ref{alg:new} when it terminates.  According to the ``while" condition in Algorithm \ref{alg:new}, we must have $\sum_{\tau=0}^{t-1} \lfloor \rho^\tau B_1\rfloor \geq T$, which further implies $\sum_{\tau=0}^{t-1} \rho^\tau B_1 \geq T$. Simplifying the partial sum of geometric series and rearranging terms yields
\begin{align}
    t \geq& \log_{\rho} \left(\frac{T(\rho-1)}{B_1} + 1\right)  \nonumber \\
    \overset{(a)}{\geq}& \log_{\rho} \left((\rho-1) \frac{T}{B_1}\right) \nonumber\\
    =& \log_{\frac{1}{\rho}} \left(\frac{B_1}{T(\rho-1)}\right) \label{eq:pf-thm-PL-eq1}
\end{align}
where (a) follows because $\rho > 1$.

By Lemma \ref{lm:PL-one-step}, for all $\tau\in\{1,2,\ldots,t-1\}$, we have
\begin{align*}
    &\mbb{E}[f(\mb{x}_{\tau+1}) - f^\ast] \nonumber\\
    \leq& (1-\nu) \mbb{E}[f(\mb{x}_{\tau}) - f^\ast] + \frac{\gamma (2-L\gamma)}{2NB_\tau}\sigma^2\\
    \overset{(a)}{\leq}& (1-\nu) \mbb{E} [f(\mb{x}_{\tau}) - f^\ast] + \frac{\gamma (2-L\gamma)\sigma^2}{N B_1}\frac{1}{\rho^{\tau-1}}
\end{align*}
where (a) follows by recalling $B_{\tau} = \lfloor \rho^{\tau-1} B_1\rfloor$ and noting $\lfloor z \rfloor > \frac{1}{2} z$ as long as $z \geq 2$.

Recursively applying the above inequality from $\tau=1$ to $\tau=t-1$ yields
\begin{align}
    &\mbb{E}[f(\mb{x}_t) -f^\ast] \nonumber\\
    \leq& (1-\nu)^{t-1} \left(f(\mb{x}_{1}) - f^\ast\right)+ \frac{\gamma (2-L\gamma)\sigma^2}{N B_1}\sum_{\tau=0}^{t-2} (1-\nu)^\tau (\frac{1}{\rho})^{t-2-\tau} \nonumber\\
    = & (1-\nu)^{t-1} \left(f(\mb{x}_{1}) - f^\ast\right)+ \frac{\gamma (2-L\gamma)\sigma^2}{N B_1}(\frac{1}{\rho})^{t-2}\sum_{\tau=0}^{t-2} ((1-\nu)\rho)^\tau \nonumber\\
    \overset{(a)}{\leq} & (1-\nu)^{t-1} \left(f(\mb{x}_{1}) - f^\ast\right)+ \frac{\gamma (2-L\gamma)\sigma^2}{N B_1}(\frac{1}{\rho})^{t-2} \frac{1}{1 - (1-\nu)\rho } \nonumber\\
    \overset{(b)}{\leq} & (1-\nu)^{\log_{\frac{1}{\rho}} \left(\frac{B_1}{T(\rho-1)}\right) } \frac{1}{1-\nu}\left(f(\mb{x}_{1}) - f^\ast)\right)+ \frac{\gamma (2-L\gamma)\sigma^2}{N} \frac{1}{1 - (1-\nu)\rho } \frac{\rho^2}{T(\rho -1) } \nonumber\\
    \overset{(c)}{=}& \Big(\frac{B_1}{T(\rho -1)}\Big)^{\log_{\frac{1}{\rho}}(1-\nu)} \frac{1}{1-\nu}\left(f(\mb{x}_{1}) - f^\ast\right) + \frac{\rho^2 \gamma (2-L\gamma)\sigma^2}{(1 - (1-\nu)\rho)(\rho - 1)}\frac{1}{NT} \nonumber \\
    = & \frac{1}{1-\nu}\left(f(\mb{x}_{1}) - f^\ast\right) \Big(\frac{B_1}{
    \rho -1}\Big)^{\log_{\frac{1}{\rho}}(1-\nu)} \frac{1}{T^{\log_{\frac{1}{\rho}}(1-\nu)}} + \frac{\rho^2 \gamma (2-L\gamma)\sigma^2}{(1 - (1-\nu)\rho)(\rho - 1)}\frac{1}{NT} \nonumber \\
    \overset{(d)}{=}&   \frac{1}{1-\nu}\left(f(\mb{x}_{1}) - f^\ast\right) \Big(\frac{B_1}{\rho -1}\Big)^{ \log_{\rho}(\frac{1}{1-\nu})} \frac{1}{T^{\log_{\rho}(\frac{1}{1-\nu})}} + \frac{\rho^2    \gamma (2-L\gamma)\sigma^2}{(1 - (1-\nu)\rho)(\rho - 1)}\frac{1}{NT} 
    	 \label{eq:pf-thm-PL-eq2}
\end{align}
where (a) follows by simplifying the partial sum of geometric series and noting that $(1-\nu)\rho < 1$; (b) follows by substituting \eqref{eq:pf-thm-PL-eq1} and noting that $0< 1-\nu < 1$ and $0<\frac{1}{\rho}<1$; (c) follows by noting that $\log_{\frac{1}{\rho}} \Big(\frac{B_1}{T(\rho -1)  }\Big) = \frac{\log_{1-\nu} \big(\frac{B_1}{T(\rho - 1)  }\big)}{\log_{1-\nu}\big(\frac{1}{\rho}\big)} = \log_{\frac{1}{\rho}}(1-\nu)\log_{1-\nu} \Big(\frac{B_1}{T(\rho -1) }\Big) =  \log_{1-\nu} \Big(\Big(\frac{B_1}{T(\rho -1)}\Big)^{\log_{\frac{1}{\rho}}(1-\nu)}\Big)$; and (d) follows from $\log_{\frac{1}{\rho}}(1-\nu) = \log_{\rho}(\frac{1}{1-\nu})$.

\end{document}